\documentclass[12pt,a4paper,oneside]{amsart}
\usepackage[left=1.1in, right=0.8in, top=1.0in, bottom=1.0in]{geometry}
\usepackage{xcolor}
\usepackage{comment}
\usepackage{graphicx}
\usepackage{setspace}
\usepackage{indentfirst} 
\usepackage{cases}
\usepackage{todonotes}
\usepackage{wrapfig}
\usepackage[toc,page]{appendix}
\usepackage{amsmath}
\usepackage{amsthm}
\usepackage{amssymb}
\usepackage{enumerate}
\usepackage{mathrsfs}
\usepackage{hyperref}

\newtheorem{thm}{Theorem}
\newtheorem*{thm*}{Theorem}

\newtheorem{example}{Example}
\newtheorem{proposition}{Proposition}
\newtheorem{lemma}[proposition]{Lemma}
\newtheorem{corollary}[proposition]{Corollary}
\newtheorem{mydef}{Definition}

\newtheorem{remark}{Remark}
\theoremstyle{definition}

\DeclareMathOperator{\CAT}{CAT}

\title{Weak topologies for unbounded nets in CAT($0$) spaces} 
\author[1]{Philip Miller}
\author[2]{Arian B\"erd\"ellima}
\author[1]{Max Wardetzky}

\address[1]{%
Institute for Numerical and Applied Mathematics, University of G\"ottingen, 37083 G\"ottingen, Germany}
\address[2]{%
Institute of mathematics, Technical University Berlin, 10623 Berlin, Germany}
\address[3]{%
Institute for Numerical and Applied Mathematics, University of G\"ottingen, 37083 G\"ottingen, Germany}

\begin{document}
\begin{abstract}
Weak topologies that yield weak convergence for \emph{bounded} sequences and nets in $\CAT(0)$ spaces have been studied in the past. We are here concerned with weak topologies that yield weak convergence of \emph{unbounded} sequences and nets. We analyze two such topologies that generalize the weak topology on Hilbert spaces and that agree with the strong topology on a  $\CAT(0)$ space if and only if the space is locally compact. 
\end{abstract} 


\maketitle

\section{introduction}
A geodesic metric space $(H,d)$ is called a $\CAT(0)$ space if every geodesic triangle in $X$ is at least as {\em slim} as its corresponding comparison triangle in the Euclidean plane. 
For a detailed study of these spaces, 
see, e.g.,~\cite{Alexander2019,Bacak,Brid,Burago}. Throughout this article, we only consider \emph{complete} $\CAT(0)$ spaces, which are often called \emph{Hadamard} spaces.

The notion of \emph{weak convergence} on $\CAT(0)$ spaces has appeared in various contexts over the past years, see, e.g.,~\cite{Bacak, Espi, Jost, Kell, Kirk, Lim, Monod, Sosov}.
Recently, Lytchak and Petrunin~\cite{lytchak2021weak} resolved some open problems concerning the existence of topologies that realize weak convergence for \emph{bounded} sequences and nets in $\CAT(0)$. We are here concerned with similar questions \emph{without} the assumption of boundedness.
\begin{mydef}\label{def:weak-conv}
Let  $(H,d)$ be a $\CAT(0)$ space, and let $x,y \in H$. Denote by $[x,y]$ the unique geodesic connecting $x$ and $y$, and let $P_{[x,y]}z$ denote the closest point projection of $z \in H$ onto $[x,y]$ with respect to $d$. A net (in particular a sequence) $(x_\alpha)\subseteq H$  {\em converges weakly} to $x\in H$ if $P_{[x,y]}x_\alpha \to x$ with respect to $d$ for all $y\in H\setminus \{x\}$. 
\end{mydef}

For Hilbert spaces the above notion of weak convergence coincides with convergence in the usual weak topology. We are concerned with topologies $\tau$ on general CAT(0) spaces that \emph{realize weak convergence of nets}, by which we mean that a net weakly converges  to $x$ if and only if it converges to $x$ with respect to $\tau$. Since any sequence is a net, a topology that realizes weak convergence of nets also realizes weak convergence of sequences.

For \emph{bounded} nets (or sequences), the above definition is identical to the one given by Lytchak and Petrunin in~\cite{lytchak2021weak}, who show the following result:

\begin{thm*}[\cite{lytchak2021weak}]
Let $(H,d)$ be a $\mathrm{CAT(0)}$ space. There exists a unique topology $\tau_{LP}$ on $H$ with the following two properties:

\begin{enumerate}[i.]
\item A sequence $(x_{n})$ converges in $H$ with respect to $\tau_{LP}$ to a point $x$ if and only if the sequence is bounded and converges to $x$ weakly in the sense of Definition~\ref{def:weak-conv}.
\item The topology $\tau_{LP}$ is sequential.
\end{enumerate}
\end{thm*}

With reference to the above theorem, recall that a topology $\tau$ on a set $X$ is called \emph{sequential} if every set $A \subseteq X$ that is $\tau$-sequentially closed is $\tau$-closed. The topology $\tau_{LP}$ can be defined as follows: 
A subset $A \subseteq H$ is $\tau_{LP}$-closed if and only if for any bounded sequence $(x_n) \subseteq A$
weakly converging to some point $x \in H$, one has that $x\in A$.

Our interest in unbounded nets (and sequences) is in parts motivated by an attempt to generalize weak topologies on Hilbert spaces to CAT(0) spaces. Since  the weak topology on an infinite dimensional Hilbert space is never sequential (see, e.g., Example~\ref{ex:Hilbert}), $\tau_{LP}$ does not agree with the usual weak topology due to Property $(ii)$ of the above theorem. 

Another source of motivation for studying the unbounded case is the following example due to Monod~\cite[Example 21]{Monod}:
\begin{example}[infinite spike]\label{ex:spike}
Consider the $\mathrm{CAT(0)}$ space $(H,d)$, which consists of the family of intervals $I_n=[0,n]$, $n\in \mathbb{N}$,  glued at the common vertex 0.   Consider the unbounded sequence $(x_n)$, where $x_n$ is the nonzero endpoint in $I_n$. Then $(x_n)$ converges weakly to $0$ since every point $y\in H\setminus\{0\}$ belongs to $I_N$ for some unique $N \in \mathbb{N}$, and $P_{[0,y]}x_n=0$ for all $n>N$. Being unbounded, however, this sequence does converge with respect to $\tau_{LP}$. 
\end{example} 
Below we construct two weak topologies on CAT(0) spaces that agree with the usual weak topology on Hilbert spaces and that also realize the convergence of the sequence in Example~\ref{ex:spike}.

\subsection{Definitions of weak topologies}
Consider a CAT(0) space $(H,d)$. Let $\Gamma(H)$ be the set of all compact geodesics in $H$.  For $\gamma\in \Gamma(H)$ and $x\in H$ denote by $P_\gamma x$ the projection of $x$ to $\gamma$. For $x,y \in H$, let $[x,y]$ denote the unique geodesic connecting $x$ to $y$.
We define 
\begin{align*}
U_{x}(y):=\{z \in H: P_{[x,y]} z \neq y\} ,
\end{align*}
and we call $U_{x}(y)$ an {\em elementary set} around $x$.  
The following definition introduces two topologies related to these elementary sets. 
\begin{mydef}
\begin{enumerate}[i.]
\item Let $\tau_w$ be the topology defined by $U\in \tau_w$ if and only if for all $x\in U$ there exist finitely many points $y_1, \ldots y_n\in  H\setminus \{x\} $ such that  
\[ U_x(y_1)\cap \ldots \cap U_x(y_n) \subseteq U. \]
\item Let $\tau_g$ be the the coarsest topology containing $U_x(y)$ for all  $x,y\in H$ with $x\neq y$.
\end{enumerate}
\end{mydef}
The construction of $\tau_{w}$ in this definition was suggested by Ba\v cak~\cite{Bacak3} and was also studied in the thesis of the second author.

Besides these topologies we also consider the so-called \emph{coconvex topology} $\tau_{co}$ introduced by Monod~\cite{Monod}. This topology is defined as the coarsest topology on $H$ such that metrically closed convex sets are $\tau_{co}$-closed. Moreover, we denote by $\tau_{d}$ the metric topology.

\subsection{Main results} 
One always has the  inclusions
\[\tau_{co}\subseteq \tau_w\subseteq \tau_g\subseteq \tau_d, \]
see Proposition~\ref{prop:relations}. In particular, for convex sets it is equivalent to be closed with respect to $\tau_{w}$, $\tau_{g}$, and $\tau_{d}$. Moreover, we show that one always has that 
$\tau_{w}\subseteq \tau_{LP}$
and  that $\tau_{w}$ and $\tau_{LP}$ agree on \emph{bounded} CAT(0) spaces, see Corollary~\ref{cor:tau-w-tau-LP} and Lemma~\ref{lem:tau-w-tau-LP}.

Due to Proposition 1.3 in~\cite{lytchak2021weak}, one cannot in general expect the existence of a topology that realizes weak convergence of nets (not even in the bounded case). However, we show the following result:  

\begin{thm}\label{thm:tau-w-unique}
If there exists a topology on $H$ that realizes weak convergence of nets, then this topology is equal to $\tau_{w}$.
\end{thm}

This theorem directly follows from Corollary~\ref{thm:uniqueness}. Existence a topology that realizes weak convergence can be characterized in terms of properties of $\tau_{w}$ as follows:

\begin{thm}\label{thm:tau-w-exist}
The following statements are equivalent: 
\begin{enumerate}
\item $\tau_w$ realizes weak convergence of nets.
\item For all $x\in H$ and $y\in H\setminus \{x\}$ the point $x$ belongs to the $\tau_w$-interior of $U_x(y)$.
\end{enumerate}
In this case, $\tau_w$ satisfies the Hausdorff property. 
\end{thm}

The proof of Theorem~\ref{thm:tau-w-exist} can be found in Section~\ref{sec:weak-top-exist}. For the case of bounded CAT(0) spaces, the Hausdorff property of $\tau_{w}$ is even sufficient for the other conditions in Theorem~\ref{thm:tau-w-exist} to hold; see Corollary~\ref{cor:bounded_space_w}. 

Notice that the second condition in Theorem~\ref{thm:tau-w-exist} is weaker than to require that elementary sets be open with respect to $\tau_{w}$. In fact, it is in general unclear  whether elementary sets are $\tau_{w}$-open, as, e.g., pointed out by Ba\v cak~\cite{Bacak3}. A necessary and sufficient condition for this to be true is that $\tau_{w}=\tau_{g}$. Moreover, we show that $\tau_{w}=\tau_{g}$ is equivalent to the realization of weak convergence of nets by $\tau_{g}$; see Theorem~\ref{thm:proper}.
Example~\ref{ex:book-triang} showcases a scenario where $\tau_{w}\neq \tau_{g}$.

An interesting instance of the equality  $\tau_w=\tau_g$ is provided by locally compact CAT(0) spaces: We show that (similar to Banach spaces), the weak topologies  $\tau_w$ and $\tau_g$  agree with the metric topology $\tau_{d}$ in this case:

\begin{thm}\label{thm:loc}
 $(H,d)$ is locally compact if and only if $\tau_w$ coincides with the metric topology. In particular, in this case we have that $\tau_w=\tau_g=\tau_d$,  and weak and strong convergence of nets agree.
\end{thm}
The proof can be found in Section~\ref{sec:loc-cpct}. 

Finally, for separable CAT($0$) spaces we show an analogue of the {Eberlein--\v Smulian Theorem}, namely that $\tau_{g}$-compactness is equivalent to $\tau_{g}$-sequential compactness. Moreover, for separable CAT($0$) spaces we show that $\tau_{g}$ is metrizable on any $\tau_{g}$-compact subset, see Theorem~\ref{thm:comp_implies_seq}.

\section{Relations \& first examples}\label{sec:first}
\subsection{Relations between topologies}
We start by showing inclusions of the topologies introduced above.

 \begin{proposition}\label{prop:relations}  
One has that $\tau_{co}\subseteq \tau_w\subseteq \tau_g\subseteq \tau_d$. 
\end{proposition}
\begin{proof}
To prove the first inclusion we assume that $C\subseteq H$ is convex and metrically closed. We show that $H\setminus C$ is $\tau_w$-open. More precisely, we prove  that
\begin{align}\label{eq:convex_complement}
U_x(P_C x)\subseteq H\setminus C
\end{align}  
for all $x\in H\setminus C$, where $P_C x$ is the metric projection of $x$ onto $C$. To this end first notice that if $z\in [x,P_C x]$, then 
\[ d(z, P_C x)=d(x, P_C x)-d(x,z)\leq d(x, P_C z)-d(x,z) \leq d(z,P_C z),\] 
which implies that $P_C x=P_C z$. Now let $c\in C$. Then \[ d(c, P_C x)= d(P_C c, P_C z)\leq d(c,z)\] for all $z\in [x,P_C x]$, since $P_C$ is nonexpansiveness in CAT($0$) spaces. Hence $P_{[x,P_C x]}c= P_C x$. This proves $C\subseteq P_{[x,P_C x]}^{-1}(\{P_C x\})$, which yields \eqref{eq:convex_complement} by taking complements. This proves $\tau_{co}\subseteq \tau_w$. \\
We now turn to the middle inclusion.
By definition of $\tau_w$ every $\tau_w$-open set $U$ is a union $U=\cup_{x\in U}U_x$ where $U_x \subseteq U$ is an intersection of finitely many elementary sets around $x$. Hence $U$ is $\tau_g$-open as $\tau_g$ consists of unions of finite intersections of elementary sets. \\ 
For the last inclusion is suffices to note that elementary sets are metrically open as they are preimages of metrically open sets under continuous projections. 
\end{proof}
\begin{corollary}[closed convex sets]\label{cor:convex}
Let $C\subseteq H$ be convex. Then $C$ is $\tau_w$-closed if and only if $C$ is $\tau_g$-closed if and only if $C$ is metrically closed.
\end{corollary} 
\begin{proof}
This is a consequence of Proposition \ref{prop:relations}, since it is straightforward to verify that every topology $\tau$ on $H$ with $\tau_{co}\subseteq \tau\subseteq \tau_d$ satisfies this equivalence.
\end{proof}

\begin{proposition}\label{prop:convex_preimage}
We have $\tau_{co}=\tau_w=\tau_g$ if and only if $U_x(y)$ is $\tau_{co}$-open for all $x,y\in H$ with $x\neq y.$ In particular, if the preimages $P_{[x,y]}^{-1}(\{y\})$ are convex for all $x,y\in H$ with $x\neq y$, then $\tau_{co}=\tau_w=\tau_g.$
\end{proposition} 
\begin{proof}
The first statements follows from the fact that  $\tau_{co}\subseteq \tau_g$ and the definition of $\tau_g$. For the second statement note that $P_{[x,y]}^{-1}(\{y\})=H\setminus U_x(y)$ is metrically closed. Hence convexity of $P_{[x,y]}^{-1}(\{y\})$ implies that $U_x(y)\in \tau_{co}.$
\end{proof}

\begin{remark}\label{rem:geod-monotone}
"The property of the nice projection onto geodesics"
(property $(N)$ for short) that was introduced in \cite{Espi} says: For all $\gamma\in \Gamma(H)$ and $x,y,m\in H$ we have $P_\gamma m \in [P_\gamma x,P_\gamma y] $ if  $m\in [x,y]$. Moreover, we consider the $(Q_4)$  condition 
\[ d(x,p)<d(x,q) \text{ and } d(y,p)<d(y,q) \quad \text{implies}\quad d(m,p)\leq  d(m,q) \ \forall m\in [x,y].\] 
This condition was studied in \cite{Kirk}. We also consider the slightly stronger $(\overline{Q}_4)$  condition 
\[ d(x,p)\leq d(x,q) \text{ and } d(y,p)\leq d(y,q) \quad \text{implies}\quad d(m,p)\leq  d(m,q) \ \forall m\in [x,y],\] 	
which was used in \cite{Kakavandi} as a sufficient condition for the realization of weak convergence.  
In~\cite[Theorem 5.7]{Espi} it is shown that any $\mathrm{CAT(0)}$ space of constant curvature satisfies the $(Q_4)$ condition.\\
Clearly, the $(\overline{Q}_4)$ condition implies the $(Q_4)$ condition. Furthermore, the $(Q_4)$ condition implies property $(N)$, see~\cite[Lemma 5.15]{Espi}. 
If property $(N)$ is satisfied, then $P_{[x,y]}^{-1}(\{y\})$ is convex for all $x,y\in H$ with $x\neq y$. Hence, due to Proposition \ref{prop:convex_preimage}, each of the conditions $(Q_4),$
$(\overline{Q}_4)$ and $(N)$ imply $\tau_{co}=\tau_w=\tau_g$, and we show that these topologies realize weak convergence of nets in this case (see Theorem~\ref{thm:proper}). 
\end{remark}

\subsection{Examples}
We provide three simple examples of unbounded CAT($0$) spaces on which the topologies $\tau_g$ and $\tau_w$ agree with the coconvex topology $\tau_{co}$. 

\begin{example}[Hilbert spaces]\label{ex:Hilbert}
If $(H,d)$ is a Hilbert space, then $\tau_{co}=\tau_w=\tau_g$ agree with the usual weak topology on $H$: The usual weak topology is finer than $\tau_g$ as for $x,y\in H$ with $x\neq y$   we have  
\[ U_x(y)= \left\{ z\in H \colon \, \mathrm{Re}\left( \langle b, z-x \rangle\right)<1 \right\}  \] 
with $b=\frac{y-x}{\|y-x\|^2}$. Therefore, $U_x(y)$ is open with respect to the usual weak topology. Moreover, the usual weak topology is generated by convex half spaces and hence coarser than $\tau_{co}$ (see also \cite[Ex. 18]{Monod}). Therefore, the claim follows from Proposition \ref{prop:relations}. \\ 
The following argument shows that the topology $\tau_{LP}$ is strictly finer than the usual weak topology on $H$ if $H$ is infinite dimensional: If $A$ is closed in the usual weak topology, then it is weakly sequentially closed. Therefore, it is $\tau_{LP}$-closed. This shows that $\tau_{LP}$ is finer than the usual weak topology. To see that it is strictly finer consider an orthonormal sequence $(e_n)\subseteq H$. The set $A:=\{ \sqrt{n} e_n \colon n\in \mathbb{N}\}$ is weakly sequentially closed, but 0 is contained in the weak closure of $A$. Hence $A$ is closed in $\tau_{LP}$, but not closed in the usual weak topology.
\end{example}

\begin{example}[infinite spike revisited]  
Consider the infinite spike from Example~\ref{ex:spike}. Elementary sets take one of two shapes: (i) half open intervals $(a, n]\subseteq I_{n}$, $a\geq 0$, and (ii) complements of closed intervals $[a, n]\subseteq I_{n}$, $a>0$. Therefore, elementary sets are $\tau_{co}$-open, and Proposition~\ref{prop:relations}  implies that $\tau_{co}=\tau_{w}=\tau_{g}$. 
Furthermore, the sequence $(x_{n})$ of nonzero endpoints converges to $0$ with respect to the latter topologies, but not with respect to $\tau_{LP}$. The fact that $\tau_{co}\subseteq \tau_{LP}$ (see \cite[Prop. 3.1]{lytchak2021weak}) then implies that $\tau_{LP}$ is strictly finer than $\tau_{co}=\tau_w=\tau_g$ on this space. 
Additionally, the infinite spike is compact and sequentially compact with respect to the topology $\tau_{co}=\tau_{w}=\tau_{g}$. 
\end{example}

\begin{example}[infinite dimensional hyperbolic space]
Consider the upper half space model of infinite dimensional (separable) real hyperbolic space $\mathbb{H^{\infty}}$. In this model, elementary sets take the form of one of the two connected components that arise when removing a half sphere that touches the boundary of $\mathbb{H^{\infty}}$ orthogonally. Therefore, elementary sets are $\tau_{co}$-open, and Proposition~\ref{prop:relations}  implies that $\tau_{co}=\tau_{w}=\tau_{g}$.
\end{example}

\section{Results on \texorpdfstring{$\tau_w$}{tauw}}\label{sec:tau-w}
This section is devoted to the topology $\tau_w$. As a main result we provide the following uniqueness statement: Whenever there exists a topology on $H$ realizing weak convergence of nets, then this topology agrees with $\tau_w$. Moreover, we give a characterization for when $\tau_w$ realizes weak convergence. Before stating these results we require some preliminaries. 
 
\subsection{Realization of weak convergence}\label{sec:weak-top-exist}
First observe that weak convergence always implies convergence with respect to $\tau_w$:
\begin{proposition}\label{prop:w-convergence}
Let $(x_\alpha)\subseteq H$ be a net and $x\in H$. Then 
 $x_\alpha \to x$ weakly implies $x_\alpha \to x$ with respect to $\tau_w$.
\end{proposition} 
\begin{proof}
Let $U$ be $\tau_w$-open with $x\in U$. Then there exists finitely many $y_1, \ldots y_n \in H\setminus \{x\}$ such that 
\[ U_x(y_1)\cap \ldots \cap U_x(y_n) \subseteq U. \] 
Weak convergence of $x_{\alpha}\to x$ implies that $P_{[x,y_i]} x_\alpha \to x$ with respect to $d$. Hence there exists $\alpha_i$ such that $x_\alpha \in U_x(y_i)$ for all $\alpha \succcurlyeq \alpha_i.$ Let $\alpha_0\succcurlyeq \alpha_i$ for $i=1,\ldots,n$, then $x_\alpha\in U$ for all $\alpha \succcurlyeq \alpha_0.$
\end{proof}
The following lemma concerns properties of elementary sets and will be used frequently in the sequel. 
\begin{lemma} \label{prop:elementary_incl}
Let $x,y\in H$ with $x\neq y$, and let $a\in [x,y]$ different from $x$ and $y$. If $z\in U_x(a)$, then $P_{[x,y]}z\in [x,a)$. In particular,  
 $U_x(a) \subseteq U_x(y)$. Moreover, $U_x(a)\cap U_y(a)= \emptyset$.
\end{lemma} 
\begin{proof}
Let $z\in U_x(a)$. Then the convex function $\phi: [x,y] \to \mathbb{R}$, $u\mapsto d(u,z)$ (see \cite[Ex. 2.2.4]{Bacak}), satisfies $\phi(P_{[x,a]}z)={d(P_{[x,a]}z,z)< d(a,z)} =\phi(a)$. Hence $\phi(a)<  \phi(u)$ for all $u\in (a,y]$. Therefore, the function $\phi$ attains its minimum in $[x,a).$ Hence, $P_{[x,y]}z\in [x,a)$, which shows the first statement. Since $[x,a)\subseteq [x,y)$, we obtain that $z\in U_x(y).$ Finally, $z\in U_x(a)\cap U_y(a)$ would imply that $\phi(P_{[x,a]})< \phi(a)$ and $\phi(P_{[a,y]}z)< \phi(a)$, contradicting the convexity of $\phi$. 
\end{proof}

The following result characterizes $\tau_w$-closed subsets as exactly the ones that are net-closed with respect to weak convergence.  
\begin{proposition}[$\tau_w$-closed sets]\label{lem:w-closed-sets}
Let $A\subseteq H$ be a subset. Then the following statements are equivalent: 
\begin{enumerate}
\item $A$ is $\tau_w$-closed. 
\item Whenever a net in $A$ converges weakly to some $x\in H$, then $x\in A$.
\end{enumerate}
\end{proposition} 
\begin{proof}
\item[$(1)\Rightarrow(2)$:] Suppose that $A$ is $\tau_w$-closed, and let $(x_\alpha)$ be a net in $A$ and $x\in H$ with $x_\alpha \to x$ weakly. Proposition \ref{prop:w-convergence} yields $x_\alpha \to x$ with respect to $\tau_w$. Therefore, $x\in A$. 
\item[$(2)\Rightarrow(1)$:] To prove the contrapositive we assume that $A$ is not $\tau_w$-closed. Then $H\setminus A$ is not $\tau_w$-open. Hence there exists $x\in H\setminus A$ such that for every finite subset $F\subseteq H\setminus \{x\}$ there exists 
\[ x_F \in \bigcap_{z\in F} U_x(z) \quad\text{with} \quad x_F \in A. \] 
The family  $\mathcal{F}$ of finite subsets of $H\setminus \{x\}$ is a directed set by set inclusion. Together with the above construction we obtain a net $(x_F)_{F \in \mathcal{F}} \subseteq A$. We claim that this net converges weakly to $x$. Let $y\in H$ with $x\neq y$ and $0<\varepsilon < d(x,y)$. Let $y_\epsilon\in [x,y]$ such that $d(x,y_\epsilon)=\varepsilon$, and consider the one point set $F_\varepsilon=\{y_\epsilon\}$. If $F\in \mathcal{F}$ with $F\supseteq F_\varepsilon$, then $x_F\in U_x(y_\epsilon)$. Lemma \ref{prop:elementary_incl} yields $P_{[x,y]} x_F \in [x,y_\varepsilon).$ Therefore, $d(x, P_{[x,y]}x_F)<\varepsilon$, proving weak convergence of $(x_F)_{F\in \mathcal{F}}$ to $x$. 
\end{proof}

With Propositions~\ref{prop:w-convergence} and \ref{lem:w-closed-sets} at hand, we deduce a statement that relates weak convergence of nets to the topology $\tau_w$. In particular, Theorem~\ref{thm:tau-w-unique} is a direct consequence. 
\begin{corollary}\label{thm:uniqueness}
\begin{enumerate}[1.]
\item Suppose that $\tau$ is a topology on $H$ such that $\tau$-convergence of nets implies weak convergence of nets. Then $\tau_w\subseteq \tau$.
\item Suppose $\tau$ is a topology on $H$ such that weak convergence of nets implies $\tau$-convergence of nets. Then $\tau\subseteq \tau_w$.
\end{enumerate}
\end{corollary}
\begin{proof}
\begin{enumerate}[1.]
\item Due to Proposition \ref{prop:w-convergence} the identity map $(H,\tau)\rightarrow (H,\tau_w)$ is continuous. 
\item Suppose $A\subseteq H$ is $\tau$-closed. Let $(x_\alpha)$ be a net in $A$ and $x\in H$ with $x_\alpha \to x$ weakly. Then $x_\alpha \to x$ with respect to $\tau$. Hence $x\in A$. Therefore the statement follows from Proposition~\ref{lem:w-closed-sets}. \qedhere
\end{enumerate}
\end{proof}

We are now in the position to complete the proof of Theorem~\ref{thm:tau-w-exist}.

\begin{proof}[Proof of Theorem~\ref{thm:tau-w-exist}]
\item[$(1)\Rightarrow(2)$:] We prove the contrapositive implication. Suppose that there exist $x\in H$ and $y\in H\setminus \{x\}$ such that $x$ does not belong to the $\tau_w$-interior of $U_x(y)$. Then for every $\tau_w$-open set $U$ with $x\in U$ there exists $x_U\in U$ with $x_U\notin U_x(y)$. The family $\mathcal{U}:=\{ U\in \tau_w\colon x\in U \}$ is directed with respect to reverse set inclusion, and the net $(x_U)_{U \in \mathcal{U}}$ converges to $x$ with respect to $\tau_w$. On the other hand $P_{[x,y]} x_U=y$ for all $U\in \mathcal{U}$.  Therefore, $(x_U)_{U\in \mathcal{U}}$ does not weakly converge to $x$. Hence $\tau_w$ does not realizes weak convergence.
\item[$(2)\Rightarrow(1)$:] 
Let $(x_\alpha)$ be a net in $H$ and $x\in H$. Due to Proposition \ref{prop:w-convergence} it suffices to show that  $x_\alpha \rightarrow x$ with respect to $\tau_w$ implies $x_\alpha \rightarrow x$ weakly. To this end let $y\in H$ with $x\neq y$ and $0<\varepsilon< d(x,y)$. Let $y_\epsilon\in [x,y]$ such that $d(x,y_\epsilon)=\varepsilon$. By $(2)$ there exists a $\tau_w$-open set $U\subseteq U_x(y_\varepsilon)$ with $x\in U$. Therefore, we obtain $x_\alpha\in U_x(y_\varepsilon)$ eventually. By Lemma \ref{prop:elementary_incl} this yields $P_{[x,y]}x_\alpha\in [0,y_\varepsilon)$ and therefore $d(x,P_{[x,y]}x_\alpha)<\varepsilon$ eventually. \\ 
For the last statement in Theorem~\ref{thm:tau-w-exist} one can argue either through uniqueness of weak limits or use the last statement in Lemma \ref{prop:elementary_incl} together with $(2)$. 
\end{proof}

We end this subsection with three statements that relate the topology $\tau_w$ to the topology $\tau_{LP}$ introduced by Lytchak and Petrunin.

\begin{corollary}\label{cor:tau-w-tau-LP}
It holds that $\tau_w\subseteq \tau_{LP}$. 
\end{corollary} 
\begin{proof}
Suppose $A\subseteq H$ is $\tau_w$-closed. Let $(x_n)\subseteq A$ be a bounded sequence weakly converging to some $x$ in $H$. Then Proposition \ref{lem:w-closed-sets} yields $x\in A$. Hence $A$ is $\tau_{LP}$-closed. 
\end{proof}

\begin{lemma}\label{lem:tau-w-tau-LP}
Let $A\subseteq H$ be a bounded subset. Then $A$ is $\tau_{LP}$-closed if and only if $A$ is $\tau_w$-closed. 
\end{lemma} 
\begin{proof}
Suppose $A$ is $\tau_{LP}$-closed. We use the characterization of $\tau_w$-closed sets in Proposition \ref{lem:w-closed-sets} to show that $A$ is $\tau_w$-closed. To this end let $(x_\alpha)\subseteq A$ be a net weakly converging to a point $x\in H.$ Then Lemma $5.2$ in \cite{lytchak2021weak} yields $x_\alpha\to x$ in $\tau_{LP}.$ Hence $x\in A$. The other implication is due to the previous Corollary.
\end{proof}

\begin{corollary}\label{cor:w-equal-LP}
If $(H,d)$ is bounded, then $\tau_w= \tau_{LP}$. In particular, $\tau_w$ is sequential for bounded Hadamard spaces.
\end{corollary}

\subsection{Weak convergence of bounded sequences and nets} 
It follows from Corollary~\ref{cor:w-equal-LP} that $\tau_w$ always realizes weak convergence of bounded sequences. Nevertheless, in order to keep our paper self-contained, we provide another proof of this fact that does not invoke results from \cite{lytchak2021weak}.  \\
Recall that a topological space is compact if and only if every net has a convergent subnet.

\begin{lemma}[compactness]\label{lem:w-compact}
If $K\subseteq H$ is bounded and $\tau_w$-closed, then $K$ is $\tau_w$-compact. In particular, every metrically closed ball is $\tau_w$-compact. 
\end{lemma}
\begin{proof}
First notice that due to \cite[Proposition 3.5.]{Kirk} every net in the bounded set $K$ has a weakly convergent subnet. Together with Proposition \ref{prop:w-convergence} this yields compactness of $K$. The second statement follows from Corollary \ref{cor:convex} since metrically closed balls are convex. 
\end{proof}

\begin{proposition} 
Let $x\in H$ and $(x_n)$ be a sequence in $H\setminus \{x\}$ . If $x_n\to x$ weakly, then 
the set \[ A:=\{ x_n \colon n\in \mathbb{N} \} \cup \{x\} \] is $\tau_w$-closed.
\end{proposition} 
\begin{proof}
Let $y$ be an element in the complement of $A$. Denote by $m$ the midpoint between $x$ and $y$. Then there exists $N\in \mathbb{N}$ such that $x_n\in U_x(m)$ for all $n > N$. Moreover, $x\in U_x(m)$. Lemma \ref{prop:elementary_incl} yields $x_n \notin U_y(m)$  for all $n>N$ and $x\notin U_y(m)$. Hence 
\[ U_y(m) \cap U_y(x_1) \cap \ldots \cap U_y(x_N) \subseteq H\setminus A.\]
Therefore, $H\setminus A$ is $\tau_w$-open. 
\end{proof}
\begin{lemma}[bounded sequences] 
Let $(x_n)$ be a bounded sequence in $H$ and $x\in H$. Then $x_n \rightarrow x$ with respect to $\tau_w$ if and only if $x_n\rightarrow x$ weakly. 
\end{lemma} 
\begin{proof}
Without loss of generality we may assume that $x_n\neq x$ for all $n\in \mathbb{N}$. 
Suppose $x_n \rightarrow x$ with respect to $\tau_w$ and let $(x_{m})$ be a subsequence of $(x_n)$.  Due to boundedness of $(x_{m})$ there exists a subsequence $(x_k)$ of $(x_{m})$ and $y\in H$, such that $x_k\to y$ weakly \cite[Section 3.]{Kirk}. By the last proposition $A:=\{ x_k \colon k\in \mathbb{N} \} \cup \{y\}$ is $\tau_w$-closed. The convergence of $(x_k)$ to $x$ with respect to $\tau_w$ shows that $x\in A$. Hence $x=y.$ Hence every subsequence of $(x_n)$ has a subsequence that weakly converges to $x$. This shows that $x_n \rightarrow x$ weakly. The converse is due to Proposition \ref{prop:w-convergence}.
\end{proof}

\begin{lemma}[bounded nets]\label{lem:bounded nets}
Suppose that $(H,\tau_w)$ is Hausdorff. Let $(x_\alpha)$ be a bounded net in $H$ and $x\in H$. Then $x_\alpha \rightarrow x$ with respect to $\tau_w$ if and only if $x_\alpha\rightarrow x$ weakly. 
\end{lemma}
\begin{proof}
Suppose $x_\alpha \rightarrow x$ with respect to $\tau_w$ and let $(x_\beta)$ be a subnet of $(x_\alpha)$. Boundedness of  $(x_\beta)$ yields the existence of a subnet $(x_\delta)$ of $(x_\beta)$ that converges weakly to some $y\in H$. Due to Proposition \ref{prop:w-convergence} we obtain $x_\delta \rightarrow y$ with respect to $\tau_w$. Moreover, $x_\delta\rightarrow x$  with respect to $\tau_w$. The Hausdorff property implies $x=y$. Hence every subnet of $(x_\alpha)$ has a subnet that weakly converges to $x$. This shows that $x_\alpha \rightarrow x$ weakly. The converse is due to Proposition \ref{prop:w-convergence}.
\end{proof}
\begin{corollary}[bounded spaces] \label{cor:bounded_space_w}
Suppose $(H,d)$ is bounded. Then the following statements are equivalent: 
\begin{enumerate}
\item $\tau_w$ realizes weak convergence.
\item For all $x\in H$ and $y\in H\setminus \{x\}$ the point $x$ belongs to the $\tau_w$-interior of $U_x(y)$.
\item $\tau_w$ satisfies the Hausdorff property. 
\end{enumerate}
\end{corollary} 
\begin{proof}
This follows directly from Theorem \ref{thm:tau-w-exist} and Lemma \ref{lem:bounded nets}. 
\end{proof}

\begin{remark}
In particular, on a bounded $\mathrm{CAT(0)}$ space the conditions of the previous corollary are satisfied if $\tau_{co}$ is Hausdorff since $\tau_{co}\subseteq \tau_{w}$. Moreover, by Lemma~\ref{lem:w-compact}, $\tau_{w}$ is compact on a bounded space; hence $\tau_{co} = \tau_{w}$ (since two topologies that are both compact and Hausdorff and where one topology contains the other must agree).
\end{remark}

\begin{remark} 
It remains open whether the boundedness assumption in Corollary \ref{cor:bounded_space_w} can be dropped or if there exists a space $H$ such that $(H,\tau_w)$ is Hausdorff but $\tau_w$ does not realize weak convergence.
\end{remark} 

\section{Results on \texorpdfstring{$\tau_g$}{taug}}

\subsection{Continuous projections}
We show that $\tau_g$ is the initial topology with respect to the projections $P_\gamma\colon (H,\tau_g)\rightarrow (\gamma,d)$ for $\gamma \in \Gamma(H)$. Moreover, we characterize spaces for which 
$\tau_g$ realizes weak convergence. 
\begin{lemma}[continuous projections] \label{lem:proj}
The projection map $P_\gamma\colon (H,\tau_g)\rightarrow (\gamma,d)$ is continuous for all $\gamma \in \Gamma(H)$. Moreover, $\tau_g$ is the coarsest topology on $H$ with this property. 
In other words, $\tau_g$ is the initial topology of the family of maps $P_\gamma\colon (H,\tau_g)\rightarrow (\gamma,d)$ for $\gamma\in \Gamma(H)$. 
\end{lemma}
\begin{proof}
Let $ \gamma=[a,b] \in \Gamma(H)$.  To show
 the first statement it suffices to show that $P_\gamma^{-1} ([x,y])$ is $\tau_g$-closed for all $x,y\in \gamma$ with $d(a,x)\leq d(a,y)$. To this end, we will show the following identity 
 \begin{align}\label{eq:preimage_of_closed_interval}
 P_\gamma^{-1} ([x,y]) = \left(H\setminus U_a(x) \right) \cap \left( H\setminus U_b(y) \right). 
  \end{align}
First let $z\in  P_\gamma^{-1} ([x,y])$. As the function $u\mapsto d(u,z)$ on $(\gamma,d)$ is convex and attains its unique minimum in $[x,y]$ we have that $d(x,z)\leq d(u,z)$ for all $u\in [a,x]$. Hence $P_{[a,x]}z=x$, i.e. $z\notin U_a(x) $. Likewise, $d(y,z)\leq d(u,z)$ for all $u\in [y,b]$ shows $P_{[b,y]}z=y$ and hence $z\notin U_b(y)$.\\
Now suppose $z$ belongs to the right hand side in \eqref{eq:preimage_of_closed_interval}. Then $P_{[a,x]}z=x$ and therefore $d(z, P_\gamma z)\leq d(z,x) < d(z,u)$ for all $u\in [a,x)$. Likewise, $P_{[b,y]}z=y$ yields $d(z, P_\gamma z)\leq  d(z,y) < d(z,u)$ for all $u\in (y,b]$. Hence 
\[ P_\gamma z\in \gamma\setminus ([a,x) \cup (y,b])= [x,y]. \]
This proves \eqref{eq:preimage_of_closed_interval} and therefore the first statement. \\ 
To prove the second statement suppose that $\tau$ is a topology on $H$, such that $P_\gamma\colon (H,\tau)\rightarrow (\gamma,d)$ is continuous for all $\gamma\in \Gamma(H)$. Let $x,y\in X$. The subset $[x,y)$ is open in $([x,y],d)$. Hence $U_x(y)=P_{[x,y]}^{-1}([x,y))\in \tau.$ Therefore, $\tau_g\subseteq \tau.$ 
\end{proof}

\begin{corollary}[Hausdorff property]\label{cor:haus}
$(H,\tau_g)$ satisfies the Hausdorff property.
\end{corollary} 
\begin{proof}
Due to Lemma \ref{lem:proj} and since $(\gamma,d)$ is Hausdorff for all $\gamma\in \Gamma(H)$, it suffices to show that the family of projections $P_\gamma\colon H \rightarrow \gamma$ separates points. To this end let $x$ and $y$ be distinct points in $H$. Then $P_{[x,y]} x=x \neq y=P_{[x,y]} y.$
\end{proof}
Another consequence is that we can describe convergence of nets in $\tau_g$ and compare it to weak convergence. 
\begin{corollary}[convergence in $\tau_g$]\label{cor:g-convergence}
Let $(x_\alpha)\subseteq H$ be a net and $x\in H$. Then $x_\alpha \to x$ with respect to $\tau_g$ if and only if $P_\gamma x_\alpha \rightarrow P_\gamma x$ with respect to $d$ for all $\gamma\in \Gamma(H)$. In particular, $x_\alpha \to x$ with respect to $\tau_g$ implies $x_\alpha \to x$ weakly.
\end{corollary}
\begin{proof}
The first statement is immediate by Lemma \ref{lem:proj} and convergence in initial topologies. For the second statement let $y\in H\setminus \{x\}$, then the first statement yields $P_{[x,y]}x_\alpha \rightarrow P_{[x,y]} x=x$ with respect to $d$. Hence $x_\alpha \to x$ weakly.
\end{proof}
The following result provides several characterizations for when the two weak topologies $\tau_w$ and $\tau_g$ agree.
\begin{thm}\label{thm:proper}
The following statements are equivalent: 
\begin{enumerate} 
\item $\tau_w=\tau_g$. 
\item For all $x,y\in H$ with $x\neq y$ the elementary set $U_x(y)$ is $\tau_w$-open.
\item $P_\gamma: (H,\tau_w) \rightarrow (\gamma,d)$ is continuous for all $\gamma\in \Gamma(H)$. 
\item Weak convergence of a net $(x_{\alpha})$ to a point $x\in H$ implies that $P_{\gamma} x_\alpha \rightarrow P_{\gamma} x$ with respect to $d$ for all $\gamma \in \Gamma(H)$.
\item $\tau_g$ realizes weak convergence of nets. 
\end{enumerate}
\end{thm}
\begin{proof}
Due to Proposition \ref{prop:relations} we always have $\tau_w\subseteq  \tau_g$. Since $\tau_g$ is generated by elementary sets the equivalence of $(1)$ and $(2)$ is evident. Moreover, we have:
\item[$(1)\Rightarrow(3)$:] See Lemma \ref{lem:proj}.
\item[$(3)\Rightarrow(4)$:] This implication follows from Proposition \ref{prop:w-convergence}.
\item[$(4)\Rightarrow(5)$:] In view of the last statement in Corollary \ref{cor:g-convergence} it suffices to show that weak net-convergence implies net-convergence in $\tau_g$. To this end let $(x_\alpha)\subseteq H$ be a net that converges weakly to $x\in H.$ By $(4)$ we obtain $P_\gamma x_\alpha \to P_\gamma x$ in $d$ for all $\gamma\in \Gamma(H)$. Hence the first statement in Corollary \ref{cor:g-convergence} yields $x_\alpha\to x$ with respect to $\tau_g$. 
\item[$(5)\Rightarrow(1)$:] See Theorem~\ref{thm:tau-w-unique}.
\end{proof}

\begin{example}[book of triangles]\label{ex:book-triang}
The following example provides an instance where  $\tau_{w}\neq \tau_{g}$.
 
Consider a Euclidean isosceles right triangle $T$. Label one of its catheti as $a$. Let $A$ denote the vertex at the right angle, let $B$ denote the other vertex on $a$, and let $C$ denote the remaining vertex. Now consider the $\mathrm{CAT(0)}$ space $H$ consisting of countably many copies of $T$ that are isometrically glued along $a$. Enumerate the copies of $T$ in $H$ by $T_{1}, T_{2}, \dots$. For each $n \in \mathbb{N}$, let $C_{n}\in T_{n}$ denote the vertex corresponding to $C\in T$. Due to glueing, the vertices $A$ and $B$ belong to every triangle $T_{n}$. 

Let $P$ denote the midpoint of the hypothenuse in $T_{1}$. Then the elementary set $U_{P}(B)$ contains $A$. 
We claim that there is no finite intersection of elementary sets around $A$  that is contained in $U_{P}(B)$. In order to see this, choose some point $D\neq A$ in $H$. Then either $D \in T_{n}$ for all $n \in \mathbb{N}$ or $D \in T_{k}$ for exactly one $k\in \mathbb{N}$.  Consider the elementary set $U_{A}(D)$. The projection of $C_{n}$ to the geodesic $[A, D]$ is equal to $A$ for for all $n \in \mathbb{N}$ except possibly for $n=k$. Hence, any finite intersection of elementary sets around $A$ contains all but finitely many of the points $C_{1}, C_{2}, \dots$. On the other hand, none of the points $C_{2}, C_{3}, \dots $ belong to $U_{P}(B)$, since their projection onto the geodesic $[P,B]$ is equal to $B$. This proves the claim. In particular, the elementary set $U_{P}(B)$ is not $\tau_{w}$-open; hence $\tau_{w}\neq \tau_{g}$.
\end{example}

\begin{remark}
Consider the compact space that consists of the first three pages of the book of triangles constructed in Example~\ref{ex:book-triang}. On this space, we have that $\tau_{co}=\tau_{w}=\tau_{g}=\tau_{d}$, since the coconvex topology agrees with the metric topology on compact spaces, see~\cite[Lemma 17]{Monod}. However, even for this simple case property $(N)$ from Remark~\ref{rem:geod-monotone} fails to hold. Indeed, the projection of $C_{2}$ and $C_{3}$ to the geodesic $[P,B]$ is equal to $B$, whereas the midpoint of $[C_{2}, C_{3}]$ is $A$, which projects to $P$.
\end{remark}

\begin{remark}
Theorem~\ref{thm:proper} suggests that the coincidence of $\tau_w$ and $\tau_g$ might be a structurally interesting property of $\mathrm{CAT}(0)$ spaces since it might perhaps serve as an analogue of reflexivity in the category of Banach spaces. For example,  Theorem~\ref{thm:loc} shows that $\tau_{w}=\tau_{g}$ on all locally compact $\mathrm{CAT}(0)$ spaces.\\ 
Nonetheless, it is not clear to us whether the coincidence $\tau_{w}=\tau_{g}$ on a $\mathrm{CAT}(0)$ space $(H,d)$ is inherited by closed convex subsets $C\subset H$. While it is not hard to see that this is true if $C$ is bounded, we are not able to prove this permanence property for unbounded $C$. In particular, it is not clear to us whether for an unbounded net $(x_\alpha)\subset C$ and $x\in C$ weak convergence $x_\alpha\rightarrow x$ in the $\mathrm{CAT}(0)$ space $(C,d)$ implies weak convergence in $(H,d)$. 
\end{remark}

\subsection{Compactness in \texorpdfstring{$\tau_g$}{taug} }
In this section we study compactness properties of $\tau_g$. We start with observing that $\tau_g$-compactness implies the equality of $\tau_g$ and $\tau_w$. 
\begin{proposition}[compactness in $\tau_g$] \label{lem:g-compact}
If $(H,\tau_g)$ is compact, then $\tau_g=\tau_w$. 
\end{proposition} 
\begin{proof}
Let $x,y\in H$ with $x\neq y$. We prove that the elementary set $U_x(y)$ is $\tau_w$-open, which implies the result due to Theorem \ref{thm:proper}.  First notice that 
\[A= H\setminus U_x(y)\] is $\tau_g$-closed. Therefore, $A$ is $\tau_g$-compact. Now let $z\in U_x(y)$. For $a\in A$ let $m_a$ be the midpoint between $z$ and $a.$ Clearly, $a\in U_a(m_a)$ and $U_a(m_a)$ is $\tau_g$-open. Moreover, $U_a(m_a)\cap U_z(m_a)=\emptyset$ (see Lemma \ref{prop:elementary_incl}). Since $U_a(m_a)$ for $a\in A$ forms an open cover of $A$ there exists finitely many $a_1,\ldots,a_n\in A$ such that $U_{a_1}(m_{a_1}),\ldots,U_{a_n}(m_{a_n})$ cover $A$. Hence 
\[ U_z(m_{a_1}) \cap \ldots \cap U_z(m_{a_n})\subseteq H\setminus A = U_x(y), \] which proves $\tau_w$-openness of $U_x(y)$.
\end{proof}
\begin{corollary}
Suppose that $(H,d)$ is bounded. Then $(H,\tau_g)$ is compact if and only if $\tau_g=\tau_w$.  
\end{corollary} 
\begin{proof}
See Proposition~\ref{lem:g-compact} and  Lemma~\ref{lem:w-compact}. 
\end{proof}

Notice that the \emph{book of triangles} constructed in Example~\ref{ex:book-triang} is bounded but cannot be $\tau_{g}$-compact due to the previous corollary.

In the sequel we show an analogue of the {Eberlein--\v Smulian Theorem} and a metrizability result for $\tau_{g}$-compact subsets of separable CAT($0$) spaces.
Aiming for an application of Tychonoff's theorem we consider the map 
\[ \Phi \colon (H,\tau_g) \rightarrow \prod_{\gamma\in \Gamma(H)}  (\gamma,d) \quad\text{given by} \quad (\Phi(x))_{\gamma}= P_\gamma x.  \]
\begin{proposition}[homeomorphism onto image]\label{prop:homeo}
Let $\Phi(H)\subseteq \prod_{\gamma\in \Gamma(H)}  (\gamma,d)$ carry the subspace topology $\tau_\textrm{im}$. Then    \[ \Phi \colon (H,\tau_g)\rightarrow (\Phi(H),\tau_\textrm{im}) \]  is a homeomorphism. 
\end{proposition}
\begin{proof}
The same argument as in the proof of Corollary \ref{cor:haus} shows that $\Phi$ is injective. We write $\pi_\gamma \colon \prod_{\gamma\in \Gamma(H)}  (\gamma,d)\rightarrow (\gamma,d)$ for the natural projections. Then due to Lemma \ref{lem:proj} the maps $\pi_\gamma \circ \Phi=P_\gamma$ are continuous. The universal property of the product spaces yields that $\Phi$ is continuous.  The continuity of the inverse $\Phi^{-1}$ follows from continuity of $P_\gamma\circ \Phi ^{-1}=\pi_\gamma$ together with the universal property of the initial topology $\tau_g$; see Lemma \ref{lem:proj}.
\end{proof}

\begin{thm} \label{thm:comp_implies_seq}
Assume that $(H,d)$ is separable. Let $K\subseteq H$ be a subset. Then $K$ is $\tau_g$-compact if and only if it is $\tau_g$-sequentially compact. In this case, the subspace topology of $\tau_g$ on $K$ is metrizable.
\end{thm}  
\begin{proof}
First assume that $K$ is $\tau_g$-compact. 
Let $D\subseteq (H,d)$ be countable and dense. Let $\Gamma_D\subseteq \Gamma(H)$ be the set of all $\gamma\in \Gamma(H)$ with both endpoints in $D$. Then $\Gamma_D$ is countable. 
Consider the map 
\[ \Phi_D \colon (H,\tau_{g}) \rightarrow \prod_{\gamma\in \Gamma_D}  (\gamma,d) \quad\text{given by} \quad (\Phi_{D}(x))_{\gamma}= P_\gamma x.  \]
The same argument as in the proof of Proposition \ref{prop:homeo} shows that $\Phi_D$ is continuous. We claim  that $\Phi_D$ is injective. To see this, let $x,y\in H$ with $x\neq y.$ We set $r=d(x,y)$. Let $\gamma$ be the geodesic connecting $x$ and $y$. There exist  \[ x_D, y_D\in D \quad\text{such that}\quad d(x,x_D)< \frac{r}{3} \quad \text{and}\quad d(y,y_D)< \frac{r}{3} .\]
Let $\gamma_D$ be the geodesic connecting $x_D$ and $y_D$.  
From $x_D\in \gamma_D$ we obtain
\[ d(P_{\gamma_D} x, x )\leq d(x_D ,x )< \frac{r}{3}.\] 
Likewise $d(P_{\gamma_D} y, y )< \frac{r}{3}$. With the triangle inequality we deduce 
\[ r= d(x,y)\leq d(x,P_{\gamma_D} x) + d(P_{\gamma_D} x,P_{\gamma_D}y ) + d(P_{\gamma_D} y,y) < \frac{2}{3}r + d(P_{\gamma_D} x,P_{\gamma_D}y ). \]
This implies  $0<\frac{1}{3}r < d(P_{\gamma_D} x,P_{\gamma_D}y )$ and we conclude that $\Phi_D$ is injective.\\ 
Since $K$ is $\tau_g$-compact and $\prod_{\gamma\in \Gamma_D}  (\gamma,d)$ satisfies the Hausdorff property $\Phi_D$ restricts to a homeomorphism  \[ \phi_D |_K \colon (K,\tau_{g})\rightarrow (\Phi_D(K) , \tau_{\textrm{K}}), \]
where $\tau_{\textrm{K}}$ denotes the subspace topology of $\Phi_D(K)$ in $\prod_{\gamma\in \Gamma_D}  (\gamma,d)$. In particular, $\tau_{g}$-compactness of $K$ implies compactness of $(\Phi_D(K) , \tau_{\textrm{K}})$. Being a countable product of metric spaces, the space $\prod_{\gamma\in \Gamma_D}  (\gamma,d)$ is metrizable. Therefore, clearly, the subspace $(\Phi_D(K) , \tau_{\textrm{K}})$ is also metrizable. Metrizability and compactness imply that $(\phi_D(K) , \tau_{\textrm{K}})$ is sequentially compact. Therefore, the subspace topology of $\tau_g$ on $K$ is metrizable and sequentially compact.\\ 
For the converse direction assume that $K$ is $\tau_g$-sequentially compact. Since $(H,d)$ is separable, so is $(K,d)$. Hence $(K,d)$ satisfies the Lindel\"off property. Since $\tau_g$ is coarser than the metric topology also $(K,\tau_g|_K)$ satisfies the Lindel\"off property. Sequential compactness implies that $(K,\tau_g|_K)$ is countably compact. Hence  $(K,\tau_g|_K)$ is compact as it is Lindel\"off and countably compact.
\end{proof}

\section{Locally compact spaces}\label{sec:loc-cpct}

We show that $\tau_w$ agrees with the metric topology for locally compact CAT($0$) spaces. We start with the following geometric lemma. 
\begin{lemma} \label{lem:cone}
Let $x\in H$, $\varepsilon>0$, $K$ be the closed ball of radius $\varepsilon$ around $x$, $P_K$ the projection to it, $y\in K$, and $m$ the midpoint between $x$ and $y$. 
If $z\in H$ satisfies 
\[ d(x,z)>  \varepsilon \quad \text{and} \quad d(y,P_K z) \leq  \frac{\varepsilon}{2}  \quad \text{then}\quad P_{[x,m]}z=m.\]
\end{lemma} 
\begin{proof}
Notice that $P_{K}z$ is the unique point on $[x,z]$ that satisfies $d(x,P_{K}z)= \varepsilon$. Hence our assumption can only be satisfied if $x\neq y$. Moreover, we have
\[ d(z, P_K z) = d(x,z)-\varepsilon .\] 
Let $q\in [x,m)$. Then 
\begin{align*}
d(y, z)  & \leq d(y,P_K z) + d(P_K z , z) \\
& \leq d(x,z) - \frac{\varepsilon}{2}\\ 
& \leq d(x, q) + d(q ,z )- \frac{\varepsilon}{2} < d(q,z).
\end{align*}
Since the function $r \mapsto d(r,z)$ is convex on $[x,y]$ and $m\in (q,y)$ the distance $d(m,z)$ is smaller or equal to a convex combination of $d(q,z)$ and $d(y,z)$. Together with the last inequality we obtain $d(m,z)<d(q,z)$. Therefore, $P_{[x,m]}z=m$.  
\end{proof}

We are now in the position to prove Theorem~\ref{thm:loc}.

\begin{proof}[Proof of Theorem~\ref{thm:loc}]
If $\tau_w$ and the metric topology coincide then closed balls are metrically compact by Lemma \ref{lem:w-compact}. In particular,  $(H,d)$ is locally compact in this case. To prove the converse let $U\subseteq H$ be metrically open and $x\in U$. Then there exists some $\varepsilon>0$ such that the closed ball $K$ of radius $\varepsilon$ around $x$ is compact and contained in $U$. Let $y_1,\ldots,y_n\in K$ be a finite $\frac{\varepsilon}{2}$-net of $K$. With $m_i$ being the midpoint between $x$ and $y_i$ Lemma \ref{lem:cone} implies that
\begin{align}\label{eq:pr_loc}
U_x(m_1) \cap \ldots\cap U_x(m_n) \subseteq K\subseteq U,
\end{align}
since $z\in H\setminus K$ implies $z\notin U_x(m_i)$ for every
 $i\in \{ 1,\ldots n\}$ that sasifies $d(y_i,P_K z)\leq \frac{\varepsilon}{2}$. The last statement of the theorem is due to Theorem~\ref{thm:proper}.
\end{proof}
\begin{remark}
Assuming that $(H,d) $ is locally compact Theorem \ref{thm:loc} yields $\tau_w=\tau_g=\tau_d$. Hence every closed ball in $(H,d)$ is metrically compact due to Lemma \ref{lem:w-compact}. This reproves the well known fact, that a Hadamard space $(H,d)$ is locally compact if and only if it is proper. (Here proper means that every closed ball in $(H,d)$ is compact). 
\end{remark}

\subsection*{Acknowledgements} We would like to thank Russell Luke and Genaro L\'opez Acedo for valuable discussions on weak convergence in CAT($0$) spaces.

 \bibliographystyle{amsplain}
  \bibliography{literature}

\providecommand{\bysame}{\leavevmode\hbox to3em{\hrulefill}\thinspace}
\providecommand{\MR}{\relax\ifhmode\unskip\space\fi MR }
\providecommand{\MRhref}[2]{%
  \href{http://www.ams.org/mathscinet-getitem?mr=#1}{#2}
}
\providecommand{\href}[2]{#2}
\begin{thebibliography}{10}

\bibitem{Alexander2019}
S.~Alexander, V.~Kapovitch, and A.~Petrunin, \emph{An {I}nvitation to
  {A}lexandrov {G}eometry: {CAT}(0) {S}paces}, Springer International
  Publishing, 2019.

\bibitem{Bacak}
M.~Ba{\v c}ak, \emph{Convex {A}nalysis and {O}ptimization in {H}adamard
  {S}paces}, vol. 22 of De Gruyter Series in Nonlinear Analysis and
  Applications, De Gruyter, Berlin, 2014.

\bibitem{Bacak3}
\bysame, \emph{Old and new challenges in hadamard spaces}, 2018,
  arXiv:1807.01355.

\bibitem{Brid}
M.~R. Bridson and A.~Haefliger, \emph{Metric {S}paces of {N}onpositive
  {C}urvature}, A Series of Comprehensive Studies in Mathematics, vol. 319,
  Springer-Verlag Berlin Heidelberg, 1999.

\bibitem{Burago}
D.~Burago, Y.~Burago, and S.~Ivanov, \emph{A {C}ourse in {M}etric {G}eometry},
  Graduate Studies in Mathematics, vol.~33, American Mathematical Society,
  2001.

\bibitem{Espi}
R.~Esp\'inola and A.~Fern\'andez-Le\'on, \emph{{CAT($\kappa$)}-spaces, weak
  convergence and fixed points}, J.Math. Anal. Appl. \textbf{353} (2009),
  410--427.

\bibitem{Jost}
J.~Jost, \emph{Equilibrium maps between metric spaces}, Calc. Var. Partial
  Diff. Equations \textbf{2} (1994), 173--204.

\bibitem{Kakavandi}
B.~A. Kakavandi, \emph{Weak topologies in complete {CAT(0)} metric spaces},
  Proc. Amer. Math. Soc. \textbf{141} (2012), 1029--1039.

\bibitem{Kell}
M.~Kell, \emph{Uniformly convex metric spaces}, Anal. Geom. Metr. Spaces
  \textbf{2} (2014), 359--380.

\bibitem{Kirk}
W.A. Kirk and B.~Panyanak, \emph{A concept of convergence in geodesic spaces},
  Nonlinear Analysis \textbf{68} (2008), 3689--3696.

\bibitem{Lim}
T.C. Lim, \emph{Remarks on some fixed point theorems}, Proc. Amer. Math. Soc.
  \textbf{60} (1976), 179--182.

\bibitem{lytchak2021weak}
A.~Lytchak and A.~Petrunin, \emph{Weak topology on cat(0) spaces}, 2021.

\bibitem{Monod}
N.~Monod, \emph{Superrigidity for irreducible lattices and geometric
  splitting}, J. Amer. Math. Soc. \textbf{19} (2008), 781--814.

\bibitem{Sosov}
E.N. Sosov, \emph{On analogues of weak convergence in a special metric space},
  Russian Math. (Iz. VUZ) \textbf{48} (2004), no.~5, 79--83.

\end{thebibliography}
\end{document}